\documentclass[a4paper]{amsart}
\usepackage{graphicx}
\usepackage{amssymb}
\usepackage{amsmath}
\usepackage{amsthm}
\usepackage{amscd}
\usepackage{color}
\usepackage{colordvi}
\usepackage[all,2cell]{xy}

\UseAllTwocells \SilentMatrices
\newtheorem*{theorem*}{Theorem A}
\newtheorem*{theorem**}{Theorem B}
\newtheorem{theorem}{Theorem}[section]
\newtheorem{corollary}[theorem]{Corollary}
\newtheorem*{corollary*}{Corollary}
\newtheorem{lemma}[theorem]{Lemma}
\newtheorem*{lemma*}{Lemma}
\newtheorem{proposition}[theorem]{Proposition}
\newtheorem*{proposition*}{Proposition}
\theoremstyle{remark}
\newtheorem{remark}[theorem]{Remark}
\newtheorem*{remark*}{Remark}
\newtheorem{example}[theorem]{Example}
\newtheorem*{example*}{Example}
\newtheorem*{observation*}{Observation}
\theoremstyle{definition}
\newtheorem{definition}[theorem]{Definition}
\newtheorem*{definition*}{Definition}

\newtheorem*{conjecture*}{Conjecture}

\numberwithin{equation}{section}

\begin{document}
\title[Gorenstein cohomological dimension and stable categories]{Gorenstein cohomological dimension and stable categories for groups}
\author[Wei Ren] {Wei Ren}

\makeatletter
\@namedef{subjclassname@2020}{\textup{2020} Mathematics Subject Classification}
\makeatother

\thanks{}
\date{}

\thanks{E-mail: wren$\symbol{64}$cqnu.edu.cn}

\maketitle

\dedicatory{}%
\commby{}%
%\begin{center}
%\end{center}

\begin{abstract}
First we study the Gorenstein cohomological dimension ${\rm Gcd}_RG$ of groups $G$ over coefficient rings $R$, under changes of groups and rings; a characterization for finiteness of ${\rm Gcd}_RG$ is given. Some results in literature obtained over the coefficient ring $\mathbb{Z}$ or rings of finite global dimension are generalized to more general cases. Moreover, we establish a model structure on the weakly idempotent complete exact category $\mathcal{F}ib$ consisting of fibrant $RG$-modules, and show that the homotopy category $\mathrm{Ho}(\mathcal{F}ib)$ is triangle equivalent to both the stable category $\underline{\mathcal{C}of}(RG)$ of Benson's cofibrant modules, and the stable module category ${\rm StMod}(RG)$. The relation between cofibrant modules and Gorenstein projective modules is discussed, and we show that under some conditions such that ${\rm Gcd}_RG<\infty$, ${\rm Ho}(\mathcal{F}ib)$ is equivalent to the stable category of Gorenstein projective $RG$-modules, the singularity category, and the homotopy category of totally acyclic complexes of projective $RG$-modules.
\end{abstract}

\section{Introduction}

The origin of Gorenstein projective dimension may date back to the study of $G$-dimension by Auslander and Bridger \cite{AB69} in 1960s. Let $G$ be any group. Over the group ring $RG$, the coefficient ring $R$ is an $RG$-module with trivial group action. The \emph{cohomological dimension} ${\rm cd}_{R}G$, and the \emph{Gorenstein cohomological dimension} ${\rm Gcd}_{R}G$, are defined as the projective dimension and Gorenstein projective dimension of the trivial $RG$-module $R$ respectively. Studying groups through their cohomological dimensions arose from both topological and algebraic sources; see \cite{Bro82}. We refer to Introduction of \cite{ET18} for more details on the geometric significance of the Gorenstein cohomological dimension of groups.

The Gorenstein cohomological dimension of groups was extensively studied under an assumption that the coefficient ring is either the ring of integers $\mathbb{Z}$ (see e.g. \cite{ABS09, BDT09, DT08, DT10, Tal14}), or a commutative ring of finite (weak) global dimension (see e.g. \cite{Bis21+, Bis21, ET14, ET18}). However, as stated in \cite[Remark 1.4(i)]{ET22}, in dealing with Gorenstein cohomological dimension of groups, the assumption that the coefficient ring has finite global dimension is perhaps unnatural. In fact, non-projective Gorenstein projective modules fail to exist over rings of finite global dimension, and it is more reasonable to consider \emph{Gorenstein regular rings} \cite{Bel00, EEGR}, i.e. the rings with finite global Gorenstein projective dimension, in Gorenstein homological theory.

Firstly, we are inspired to consider Gorenstein cohomological dimension of groups over more general rings. Specifically, some important results over $\mathbb{Z}$ or (noetherian) rings of finite global dimension in the literature are generalized; see for example Theorem \ref{thm:fGcd}, Theorem \ref{thm:H2Gequ}, Proposition \ref{prop:gd-bound} and Corollary \ref{cor:Gcd=pdB1}.

The main result of Section 2 is Theorem \ref{thm:fGcd}. We characterize the finiteness of ${\rm Gcd}_RG$ when the coefficient ring $R$ is Gorenstein regular; this generalizes \cite[Theorem 2.7]{BDT09}, \cite[Theorem 6.4]{ET14} and \cite[Theorem 1.7]{ET18}. There are some cases for ${\rm Gcd}_RG <\infty$; see Examples \ref{eg:fGcd}, \ref{eg:H1Fgroup} and Corollary \ref{cor:Gcd=pdB2}. It is worth to note that Corollary \ref{cor:Gcd=pdB2} supports a question raised in \cite[Conjecture 3.4]{Bis21+}; for a Gorenstein regular ring $R$, the upper bound of the Gorenstein projective dimension of all $RG$-modules is determined by ${\rm Gcd}_RG$; see Proposition \ref{prop:gd-bound} and Example \ref{eg:12}.

The ``Gcd'' can be considered as an assignment of invariants for the pairs of groups and coefficient rings $(G, R)$. In Section 3 and 4, we study the properties of Gcd under changes of groups and coefficient rings, respectively. we define an order $(H, S)\leq (G, R)$ for such pairs; see Definition \ref{def:order}. Inspired by \cite{ET18, Tal14}, we use Theorem \ref{thm:fGcd} to show that when $R$ is a commutative Gorenstein regular ring, $\mathrm{Gcd}_{R}H\leq \mathrm{Gcd}_{R}G$ if $(H, R)\leq (G, R)$, and $\mathrm{Gcd}_{S}G\leq \mathrm{Gcd}_{R}G$ if $(G, S) \leq (G, R)$; see Proposition \ref{prop:GroupOrd} and Proposition \ref{prop:RingOrd}. These generalize \cite[Propositions 2.1 and 2.4]{ET18} and \cite[Theorem 3.2]{Tal14}. Consequently, ``Gcd'' preserves the order of pairs of groups and commutative Gorenstein regular rings; see Corollary \ref{cor:KpOrd}.

By Serre's Theorem, ${\rm cd}_\mathbb{Z}G = {\rm cd}_\mathbb{Z}H$ for a torsion-free group $G$ and any subgroup $H$ of finite index; see e.g. \cite[Theorem VIII 3.1]{Bro82}. We have a Gorenstein version of Serre's Theorem for any group and its subgroups of finite index over any coefficient ring; see Theorem \ref{thm:H2Gequ}. Note that the equality ${\rm Gcd}_RG = {\rm Gcd}_RH$ was stated in \cite[Corollary 2.10]{ET18} under assumptions that the coefficient ring $R$ is of finite weak global dimension and $H$ is a normal subgroup of $G$.

Let $B(G, R)$ be an $RG$-module given by functions from $G$ to $R$ which take finitely many values \cite{CK97}. Recall that an $RG$-module $M$ is \emph{cofibrant} if the $RG$-module $M\otimes_{R}B(G, R)$ via diagonal action of $G$ is projective; see \cite[Definition 4.1]{Ben97}. It follows from \cite{CK97} that any Benson's cofibrant module is Gorenstein projective. Inspired by \cite{CK97, DT10}, in Section 5 we further study the relation between cofibrant modules and Gorenstein projective modules over group rings, and compare the Gorenstein projective dimension of the $RG$-module $M$ and the projective dimension of the $RG$-module $M\otimes_R B(G,R)$; see Propositions \ref{prop:GP-Cof1} and \ref{prop:GP-Cof2}. In particular, by letting $M$ be the trivial $RG$-module $R$, we have an equality ${\rm Gcd}_{R}G = {\rm pd}_{RG}B(G, R)$ in two cases; see Corollaries \ref{cor:Gcd=pdB1} and \ref{cor:Gcd=pdB2}. The first one generalizes \cite[Theorem 1.18]{Bis21} from $R$ being of finite global dimension to Gorenstein regular, and the second one, as stated above, supports a question in \cite[Conjecture 3.4]{Bis21+}.

In Section 6, we are devoted to study model structure and stable categories with respect to cofibrant modules. The notion of a model category was introduced by Quillen \cite{Qui67} as an axiomatization of homotopy theory. Let $\mathcal{F}ib$ consist of $RG$-modules $M$ such that ${\rm pd}_{RG}M\otimes_{R}B(G, R) < \infty$, named {\em fibrant modules}. It is clear that projective $RG$-modules and cofibrant modules are fibrant. By Lemma \ref{lem:F}, the category of fibrant modules $\mathcal{F}ib$ is a weakly idempotent complete exact category, i.e. an exact category in which every split monomorphism has a cokernel and every split epimorphism has a kernel. Moreover, we establish a model structure on $\mathcal{F}ib$; see Theorem \ref{thm:model} and compare [6, Section 10].

For model category $\mathcal{F}ib$, the associated homotopy category $\mathrm{Ho}(\mathcal{F}ib)$ is obtained by formally inverting the weak equivalences. It is standard that $\mathrm{Ho}(\mathcal{F}ib)$ is triangle equivalent to the stable category $\underline{\mathcal{C}of}(RG)$; see Corollary \ref{cor:equ1}. Moreover, ${\rm Ho}(\mathcal{F}ib)$ is equivalent to the {\em stable module category} ${\rm StMod}(RG)$; see Theorem \ref{thm:stable}. It is worth to remark that the modules of ${\rm StMod}(RG)$ were assumed to be countably presented in \cite{Ben97}, and Benson suggested that one can lift the countability hypothesis based on \cite[Open Problem 2]{Ben97}. In fact, the objects of ${\rm StMod}(RG)$ are precisely the fibrant $RG$-modules, and the countability assumption of modules in ${\rm StMod}(RG)$ is not necessary.

In \cite[Theorem 3.10]{MS19}, some equivalences of triangulated categories
$${\rm StMod}(RG) \simeq \underline{\mathcal{GP}}(RG) \simeq {\rm D}_{sg}(RG) \simeq {\rm K}_{tac}(RG\text{-}{\rm Proj})$$
are established if $G$ is a group of type $\Phi_R$ and $R$ is assumed to be a commutative noetherian ring with finite global dimension, where $\underline{\mathcal{GP}}(RG)$ is the stable category of Gorenstein projective $RG$-modules, ${\rm D}_{sg}(RG)$ is the singularity category \cite{Buc87, Or04}, and ${\rm K}_{tac}(RG\text{-}{\rm Proj})$ is the homotopy category of totally acyclic complexes of projective $RG$-modules. This can be generalized and extended as a direct application of the above. Let $R$ be a commutative ring with finite global dimension, which is not necessarily noetherian.  If $G$ is either a group of type $\Phi_R$ or an ${\rm LH}\mathfrak{F}$-group of type $FP_\infty$, then ${\rm Gcd}_RG$ is finite, and furthermore, we prove in Corollary \ref{cor:triequ} that
$$\begin{aligned} {\rm Ho}(\mathcal{F}ib) &\simeq  {\rm StMod}(RG) \simeq \underline{\mathcal{C}of}(RG) = \underline{\mathcal{GP}}(RG) \\
 &\simeq {\rm D}_{sg}(RG) \simeq {\rm K}_{tac}(RG\text{-}{\rm Proj}) = {\rm K}_{ac}(RG\text{-}{\rm Proj}),
\end{aligned}$$
where ${\rm K}_{ac}(RG\text{-}{\rm Proj})$ denotes the homotopy category of acyclic complexes of projective $RG$-modules.

\section{Finiteness of Gorenstein cohomological dimension of groups}

Throughout, all rings are assumed to be associative with unit, and all modules will be left modules unless otherwise specified.

Let $A$ be a ring. An acyclic complex of projective $A$-modules
$$\mathbf{P} = \cdots\longrightarrow P_{n+1}\longrightarrow P_{n}\longrightarrow P_{n-1}\longrightarrow\cdots$$
is said to be \emph{totally acyclic}, if it remains acyclic after applying $\mathrm{Hom}_{A}(-, P)$ for any projective $A$-module $P$. A module is \emph{Gorenstein projective} \cite{EJ00} if it is isomorphic to a syzygy of such a totally acyclic complex. For finitely generated Gorenstein projective modules, there are different terminologies in the literature, such as modules of $G$-dimension zero, maximal Cohen-Macaulay modules and totally reflexive modules; see for example \cite{AB69, Buc87, Chr00, EJ00}.

For any module $M$, the \emph{Gorenstein projective dimension} is defined in the standard way by using resolutions of Gorenstein projective modules, i.e.
$\mathrm{Gpd}_AM\leq n$ if and only if there is an exact sequence $0\rightarrow M_{n}\rightarrow M_{n-1} \rightarrow \cdots \rightarrow M_{0}\rightarrow M\rightarrow 0$ with each $M_{i}$ being Gorenstein projective.

For any ring $A$, the {\em (left) Gorenstein global dimension} of $A$ is defined as $${\rm G.gldim}(A)={\rm sup}\{{\rm Gpd}_A(M)\; |\; M\in {\rm Mod}(A)\}.$$
The preference of the terminology is justified by \cite[Theorem 1.1]{BM10} which shows that ${\rm G.gldim}(A)$ equals to the supremum of the Gorenstein injective dimension of all $A$-modules.

According to a classical result established by Serre, and by Auslander and Buchsbaum, a commutative noetherian local ring $A$ is {\em regular} if and only if its global dimension ${\rm gldim}(A)$ if finite.  Analogously, a ring with  ${\rm G.gldim}(A)<\infty$ is said to be {\em (left) Gorenstein regular}. It is also called {\em left Gorenstein} in \cite[Section 6]{Bel00}, which is equivalent to the fact that the category of left modules over the ring is a {\em Gorenstein category} in the sense of \cite[Definition 2.18]{EEGR} or \cite[Section 4]{Bel00}.

For any ring $A$, it is clear that ${\rm G.gldim}(A)\leq {\rm gldim}(A)$. Then the rings of finite global dimension are Gorenstein regular. The converse may not be true in general. For example,  both $\mathbb{Z}_4 = \mathbb{Z}/4\mathbb{Z}$ and the truncated polynomial ring $k[x]/(x^n)$ ($n\geq 2$) over a field $k$, are rings with the Gorenstein global dimension zero, while their global dimensions are infinity.

For any pair $(G, R)$ of a group $G$ and a coefficient ring $R$, the \emph{(left) Gorenstein cohomological dimension} of $G$ over $R$, denoted by $\mathrm{Gcd}_{R}G$, is defined to be the Gorenstein projective dimension of the $RG$-module $R$, where the group action on $R$ is trivial. In this section, we intend to study the finiteness of Gorenstein cohomological dimension of groups. We begin with the following observation.

\begin{lemma}\label{lem:Gp}
Let $R$ be a left Gorenstein regular ring and $G$ a group. Then any Gorenstein projective $RG$-module is also a Gorenstein projective $R$-module.
\end{lemma}

\begin{proof}
Let $M$ be a Gorenstein projective $RG$-module. Then there is a totally acyclic complex of projective $RG$-modules $\cdots\rightarrow P_{1}\rightarrow P_{0}\rightarrow P_{-1}\rightarrow\cdots$  such that
$M\cong \mathrm{Ker}(P_0\rightarrow P_{-1})$. Since any projective $RG$-module is also $R$-projective, by restricting this totally acyclic complex, we get an acyclic complex of projective $R$-modules. Noting that $R$ is left Gorenstein regular, every projective $R$-module $P$ has finite injective dimension. Apply
${\rm Hom}_R(-, P)$ to the above complex, and it is easy to see that the acyclic complex of projective $R$-modules is totally acyclic by induction on the injective dimension of $P$. Hence, $M$ is also a Gorenstein projective $R$-module.
\end{proof}

\begin{lemma}\label{lem:SplitMonic}
Let $G$ be a group, $R$ be a left Gorenstein regular ring. If ${\rm Gcd}_{R}G$ is finite, then there exists an $R$-split $RG$-exact sequence $0\rightarrow R\rightarrow \Lambda$, where $\Lambda$ is an $R$-projective $RG$-module such that ${\rm Gcd}_RG = {\rm pd}_{RG}\Lambda$.
\end{lemma}

\begin{proof}
Let ${\rm Gcd}_{R}G  = {\rm Gpd}_{RG}R = n$. It follows from \cite[Theorem 2.10]{Hol04} that there exists an exact sequence $0\rightarrow K\rightarrow M\rightarrow R\rightarrow 0$, where $M$ is a Gorenstein projective $RG$-module, and ${\rm pd}_{RG}K = n-1$. For $M$, there is an exact sequence of $RG$-modules
$0\rightarrow M\rightarrow P\rightarrow L\rightarrow 0$, where $L$ is Gorenstein projective and $P$ is projective. We consider the following pushout of $M\rightarrow R$ and $M\rightarrow P$:
$$\xymatrix@C=20pt@R=20pt{ & & 0\ar[d] & 0\ar[d] \\
0 \ar[r] &K \ar@{=}[d] \ar[r] & M \ar[d]\ar[r] &R \ar[d]\ar[r] &0 \\
0 \ar[r] &K \ar[r] & P \ar[r] \ar[d] &\Lambda \ar[r]\ar[d] & 0\\
&  & L \ar[d] \ar@{=}[r] & L\ar[d]\\
&  & 0 & 0
  }$$
From the middle row we infer that ${\rm pd}_{RG}\Lambda = {\rm pd}_{RG}K + 1 = n$. It follows from Lemma \ref{lem:Gp} that $L$ is also a Gorenstein projective $R$-module, and then the sequence
$0\rightarrow R\rightarrow \Lambda\rightarrow L\rightarrow 0$ is $R$-split. Moreover, as an $R$-module, $\Lambda\cong L\oplus R$ is Gorenstein projective. By \cite[Proposition 10.2.3]{EJ00}, which says that projective dimension of any Gorenstein projective module is either zero or infinity, we induce from ${\rm pd}_{R}\Lambda \leq {\rm pd}_{RG}\Lambda = n$ that $\Lambda$ is a projective $R$-module. This completes the proof.
\end{proof}

The above $R$-projective $RG$-module $\Lambda$ is called a {\em characteristic module} for $G$ over $R$; see e.g. \cite{Tal17} or \cite[Definition 1.1]{ET22}. As developed in \cite{BDT09, ET14, ET18, Tal17}, this notion provides a useful tool for characterizing the finiteness of ${\rm Gcd}_RG$.

\begin{lemma}\label{lem:RG-Gp}
Let $G$ be a group, $R$ be a commutative Gorenstein regular ring. If there exists an $R$-split monomorphism of $RG$-modules $\iota: R\rightarrow \Lambda$, where $\Lambda$ is $R$-projective with $\mathrm{pd}_{RG}\Lambda < \infty$, then for any $RG$-module $M$, one has
$$\mathrm{Gpd}_{RG}M \leq \mathrm{pd}_{RG}\Lambda + {\rm Gpd}_RM.$$
\end{lemma}

\begin{proof}
Let $\mathrm{pd}_{RG}\Lambda = n$. It suffices to assume that ${\rm Gpd}_RM = m$ is finite. If $M$ is Gorenstein projective as an $R$-module, i.e. ${\rm Gpd}_RM = 0$, the assertion is exactly \cite[Proposition 2.3]{ET}.

We now assume that $m > 0$ and consider a short exact sequence of $RG$-modules $0\rightarrow K\rightarrow P\rightarrow M\rightarrow 0$, where $P$ is projective. Since $P$ is restricted to be a projective $R$-module, as an $R$-module we have ${\rm Gpd}_{R}K = m-1$. Invoking the induction hypothesis, we may conclude that ${\rm Gpd}_{RG}K \leq n + (m-1)$, and hence ${\rm Gpd}_{RG}M\leq n + m$.
\end{proof}

The following generalizes \cite[Theorem 6.4]{ET14} and \cite[Theorem 1.7]{ET18}, where the coefficient ring is assumed to be a commutative (noetherian) ring with finite global dimension. Also, we recover \cite[Theorem 2.7]{BDT09} by specifying $R = \mathbb{Z}$.

\begin{theorem}\label{thm:fGcd}
Let $G$ be a group, $R$ a commutative Gorenstein regular ring. The following are equivalent:
\begin{enumerate}
\item ${\rm Gcd}_{R}G$ is finite.
\item There exists an $R$-split $RG$-exact sequence $0\rightarrow R\rightarrow \Lambda$, where $\Lambda$ is an $R$-projective $RG$-module of finite projective dimension.
\item Any $RG$-module has finite Gorenstein projective dimension.
\item $RG$ is a Gorenstein regular ring.
\end{enumerate}
\end{theorem}

\begin{proof}
The implication (1)$\Rightarrow$(2) follows from Lemma \ref{lem:SplitMonic}, and the implication (2)$\Rightarrow$(3) follows from Lemma \ref{lem:RG-Gp}. The implications (3)$\Rightarrow$(4) and (4)$\Rightarrow$(1) are obvious.
\end{proof}

There are many groups which have finite Gorenstein cohomological dimension; see Corollary \ref{cor:Gcd=pdB2} and the following examples.

\begin{example}\label{eg:fGcd}
It is well known that for any finite group $G$ and any coefficient ring $R$, ${\rm Gcd}_{R}G = 0$; see \cite[Proposition 2.19]{ABS09} for $R = \mathbb{Z}$, \cite[Corollary 2.3]{ET18} for any commutative ring $R$, and Proposition \ref{prop:fgroup} below for any associative ring $R$.

It follows from \cite{Sta68, Swan69} that $G$ is a non-trivial free group if and only if ${\rm cd}_{\mathbb{Z}}G = 1$. In this case, ${\rm Gcd}_{\mathbb{Z}}G = 1$ and $G$ acts on a tree with finite vertex stabilizers; see \cite[Theorem 3.6]{BDT09}.
\end{example}

\begin{example}\label{eg:H1Fgroup}
Let $G$ be the group which admits a finite dimensional contractible $G$-CW-complex with finite stabilizers. It follows from \cite[Lemma 2.21]{ABS09} that ${\rm Gcd}_\mathbb{Z}G$ is finite. Moreover, by \cite[Proposition 2.1]{ET18} (see also Corollary \ref{cor:GcdZG}) we have ${\rm Gcd}_RG \leq {\rm Gcd}_\mathbb{Z}G < \infty$ for any commutative ring $R$. Groups of finite virtual cohomological dimension, including polycyclic-by-finite groups and arithmetic groups, admit a finite dimensional contractible $G$-CW-complex with finite stabilizers; see \cite{CK96}.
\end{example}

\section{Gorenstein cohomological dimension under changes of groups}

In this section, we intend to compare the Gorenstein cohomological dimensions of different groups.

Let $\mathcal{S}$ be a collection of pairs $(G, R)$ of groups and coefficient rings. We define a partial order on $\mathcal{S}$ as follows. One might consider ``Gcd'' as an assignment of invariants for such pairs.

\begin{definition}\label{def:order}
Let $G$ and $H$ be groups, $R$ and $S$ be coefficient rings. For the pairs of groups and rings, we have $(H, S)\leq (G, R)$ provided that $H$ is a subgroup of $G$ and $S$ is an extension ring of $R$; symbolically, $H\leq G$ and $S\geq R$.
\end{definition}

The order seems to be better understood if we formally consider the pairs $(G, R)$ as fractions, where $G$ plays as a ``numerator'' and $R$ as a ``denominator''.

If the coefficient ring is a given commutative Gorenstein regular ring, by applying Theorem \ref{thm:fGcd} we get the following, which implies that the assignment ``Gcd'' preserves the order of pairs in this case. Note that this generalizes the result stated in \cite[Proposition 2.4]{ET18} where the coefficient ring is of finite weak global dimension; an analogous proof is included below for the sake of completeness.

\begin{proposition}\label{prop:GroupOrd}
Let $R$ be a commutative Gorenstein regular ring, $G$ be a group. For any subgroup $H$ of $G$ $($i.e. $(H, R)\leq (G, R))$, we have $\mathrm{Gcd}_{R}H\leq \mathrm{Gcd}_{R}G$.
\end{proposition}

\begin{proof}
The inequality is obvious if $\mathrm{Gcd}_{R}G$ is infinity, so it suffices to assume that $\mathrm{Gcd}_{R}G = n$ is finite. In this case, it follows from Lemma \ref{lem:SplitMonic} that there exists an $R$-split monomorphism of $RG$-modules $\iota: R\rightarrow \Lambda$, where $\Lambda$ is $R$-projective and $\mathrm{pd}_{RG}\Lambda = n$.  For any subgroup $H$ of $G$, there is an extension of group rings $RH\rightarrow RG$, which makes every $RG$-module to be an $RH$-module. Then, $\iota: R\rightarrow \Lambda$ can also be considered as an $R$-split monomorphism of $RH$-modules. Moreover, every projective $RG$-module is also projective as an $RH$-module, then $\mathrm{pd}_{RH}\Lambda \leq \mathrm{pd}_{RG}\Lambda = n$. For the trivial $RH$-module $R$, we infer from Lemma \ref{lem:RG-Gp} that
${\rm Gcd}_{R}H = \mathrm{Gpd}_{RH}R \leq \mathrm{pd}_{RH}\Lambda \leq n$. This completes the proof.
\end{proof}

By Serre's theorem, if $G$ is a torsion-free group and $H$ is a subgroup of finite index, then $\mathrm{cd}_{\mathbb{Z}}H = \mathrm{cd}_{\mathbb{Z}}G$; see e.g. \cite[Theorem 9.2]{Swan69} or \cite[Theorem VIII 3.1]{Bro82}. The following result might be regarded as a Gorenstein version of Serre's theorem. We remark that the equality for the coefficient ring $\mathbb{Z}$ holds by \cite[Proposition 2.18]{ABS09}, and the result was also proved in \cite[Corollary 2.10]{ET18} under the assumptions that the coefficient ring $R$ is of finite weak global dimension and $H$ is a normal subgroup of $G$.

\begin{theorem}\label{thm:H2Gequ}
Let $G$ be a group, $R$ be a ring of coefficients. For any subgroup $H$ of $G$ with finite index, there is an equality $\mathrm{Gcd}_{R}H = \mathrm{Gcd}_{R}G$.
\end{theorem}

\begin{proof}
By a standard argument, we infer that any Gorenstein projective $RG$-module $M$ can be restricted to be a Gorenstein projective $RH$-module. Let $\mathbf{P} = \cdots\rightarrow P_1\rightarrow P_0\rightarrow P_{-1}\rightarrow\cdots$ be a totally acyclic complex of projective $RG$-modules such that $M \cong \mathrm{Ker}(P_0\rightarrow P_{-1})$. Then, $\mathbf{P}$ is still an acyclic complex of projective $RH$-modules. Since the index $|G:H|$ is finite, there is an equivalence of functors $\mathrm{Ind}_H^G \simeq \mathrm{Coind}_H^G$; see for example \cite[Proposition III 5.9]{Bro82}. For any projective $RH$-module $Q$, there is an isomorphism ${\rm Hom}_{RH}(\mathbf{P}, Q)\cong {\rm Hom}_{RG}(\mathbf{P}, \mathrm{Ind}_H^GQ)$, and then we infer that $\mathbf{P}$ remains acyclic after applying ${\rm Hom}_{RH}(-, Q)$. Consequently, $M$ is restricted to be a Gorenstein projective $RH$-module. This fact leads to the inequality ${\rm Gpd}_{RH}N\leq {\rm Gpd}_{RG}N$ for any $RG$-module $N$. In particular, for the trivial $RG$-module $R$, we have $\mathrm{Gcd}_{R}H \leq \mathrm{Gcd}_{R}G$.

It remains to prove ${\rm Gcd}_{R}G\leq {\rm Gcd}_{R}H$. Since the inequality obviously holds if ${\rm Gcd}_{R}H = \infty$, it suffices to  assume ${\rm Gcd}_{R}H = n$ is finite. Take an exact sequence $0\rightarrow K\rightarrow P_{n-1}\rightarrow \cdots\rightarrow P_0\rightarrow R\rightarrow 0$ of $RG$-modules with each $P_i$ projective. Since $P_i$ are restricted to be projective $RH$-modules, by ${\rm Gcd}_{R}H = n$ we infer that $K$ is a Gorenstein projective $RH$-module. For the required inequality, it suffices to show that $K$ is a Gorenstein projective $RG$-module.

Let $P$ be any projective $RG$-module, which is also restricted to be a projective $RH$-module. Since $P$ is a direct summand of ${\rm Ind}_H^GP$, for any $i>0$, we infer from the isomorphism
$${\rm Ext}^i_{RG}(K, {\rm Ind}_H^GP)\cong {\rm Ext}^i_{RH}(K, P) = 0$$
that ${\rm Ext}^i_{RG}(K, P) = 0$.

Let $\alpha: K\rightarrow {\rm Ind}_H^GK$ be the composition of the $RG$-map $K\rightarrow {\rm Coind}_H^GK = {\rm Hom}_{RH}(RG, K)$ given by $k\rightarrow (rg\rightarrow rgk)$, followed by the isomorphism ${\rm Coind}_H^GK\rightarrow {\rm Ind}_H^GK$. Then $\alpha$ is an $RG$-monic and is split as an $RH$-map. Let $\beta:{\rm Ind}_H^GK\rightarrow K $ be the $RH$-map such that $\beta\alpha = {\rm id}_K$.
It follows from \cite[Lemma 2.6]{ET18} that for the Gorenstein projective $RH$-module $K$, the $RG$-module ${\rm Ind}_H^GK\cong {\rm Coind}_H^GK$ is Gorenstein projective. Then, there is an exact sequence of $RG$-modules $0\rightarrow{\rm Ind}_H^GK \stackrel{f}\rightarrow P_0\rightarrow L\rightarrow 0$, where $P_0$ is projective and $L$ is Gorenstein projective. Hence, we obtain an exact sequence of $RG$-modules $0\rightarrow K \stackrel{d}\rightarrow P_0\rightarrow L_0\rightarrow 0$, where $d = f\alpha$ and $L_0 = {\rm Coker}d$.

Let $Q$ be any projective $RH$-module, and $\gamma: K\rightarrow Q$ be any $RH$-map. Consider the following diagram
$$\begin{xymatrix}@C=20pt{
0 \ar[r] &K \ar[dd]_{\gamma} \ar[rd]^{\alpha} \ar[rr]^{d} & &P_0 \ar@{-->}@/^2pc/[lldd]^{\exists \delta} \ar[r] & L_0 \ar[r] &0\\
& & {\rm Ind}_H^G K \ar[ru]^{f}  \ar[ld]_{\gamma\beta}\\
&Q
}\end{xymatrix}$$
We have ${\rm Ext}^1_{RH}(L, Q)=0$ since $L = {\rm Coker}f$ is a Gorenstein projective $RH$-module, and then for the $RH$-map $\gamma\beta$, there exists a map $\delta:P_0\rightarrow Q$ such that $\gamma\beta = \delta f$. Moreover, $\delta d = \delta f\alpha = \gamma\beta\alpha = \gamma$. This implies that $d^*: \mathrm{Hom}_{RH}(P_{0},Q)\rightarrow\mathrm{Hom}_{RH}(K, Q)$ is epic. Hence, we infer from the exact sequence $$\mathrm{Hom}_{RH}(P_{0},Q)\longrightarrow\mathrm{Hom}_{RH}(K, Q)\longrightarrow\mathrm{Ext}_{RH}^{1}(L_{0},Q)\longrightarrow 0$$
that $\mathrm{Ext}_{RH}^{1}(L_{0},Q)=0$. Consider $0\rightarrow K \stackrel{d}\rightarrow P_0\rightarrow L_0\rightarrow 0$ as an exact sequence of $RH$-modules, where $K$ is Gorenstein projective and $P_{0}$ is projective. By \cite[Corollary 2.11]{Hol04}, $L_{0}$ is a Gorenstein projective $RH$-module. Hence, for any projective $RG$-module $P$, we infer from ${\rm Ext}^1_{RG}(L_0, {\rm Ind}_H^GP)\cong {\rm Ext}_{RH}^1(L_0, P) = 0$ that ${\rm Ext}^1_{RG}(L_0, P) = 0$. This implies that the sequence $0\rightarrow K\rightarrow P_0\rightarrow L_0\rightarrow 0$ remains exact after applying ${\rm Hom}_{RG}(-, P)$.

Then, repeat the above argument for $L_0$, we will obtain inductively an acyclic complex
$0\rightarrow K\rightarrow P_0\rightarrow P_{-1}\rightarrow P_{-2}\rightarrow \cdots$
with each $P_i$ a projective $RG$-module, which remains acyclic after applying ${\rm Hom}_{RG}(-, P)$ for any projective $RG$-module $P$. Recall that ${\rm Ext}^i_{RG}(K, P) = 0$. Consequently, we infer from \cite[Proposition 2.3]{Hol04} that $K$ is a Gorenstein projective $RG$-module, as required.
\end{proof}

We can give a homological characterization for finite groups as follows, which generalizes \cite[Proposition 2.19]{ABS09}, the main theorem of \cite{DT08}, and \cite[Corollary 2.3]{ET18}. For a ring $A$, $\mathrm{spli}(A)$ denotes the supremum of the projective lengths (dimensions) of injective $A$-modules, and $\mathrm{silp}(A)$ denotes the supremum of the injective lengths (dimensions) of projective $A$-modules; these two invariants were introduced by Gedrich and Gruenberg \cite{GG87} in connection with the existence of complete cohomological functors in the category of $A$-modules. The finitistic dimension of $A$, denoted by $\mathrm{fin.dim}(A)$, is defined as the supremum of the projective dimensions of those modules that have finite projective dimension.

\begin{proposition}\label{prop:fgroup}
Let $G$ be a group. Then the following are equivalent:
\begin{enumerate}
\item $G$ is a finite group.
\item For any coefficient ring $R$, ${\rm Gcd}_RG = 0$.
\item ${\rm Gcd}_{\mathbb{Z}}G = 0$.
\item For any left Gorenstein regular ring $R$, we have
$${\rm spli}(RG) = {\rm silp}(RG) = {\rm silp}(R) = {\rm spli}(R) = {\rm fin.dim}(R).$$
\item For any left hereditary ring $R$, ${\rm spli}(RG) = {\rm silp}(RG) = 1$.
\item ${\rm spli}(\mathbb{Z}G) = {\rm silp}(\mathbb{Z}G) = 1$.
\end{enumerate}
\end{proposition}

\begin{proof}

(1)$\Longrightarrow$(2) We consider the subgroup $H = \{1\}$ of $G$. Then ${\rm Gcd}_{R}G = {\rm Gcd}_{R}H = 0$ follows immediately from Theorem \ref{thm:H2Gequ}.

(2)$\Longrightarrow$(3) is trivial, and (1)$\Longleftrightarrow$(3) follows from \cite[Corollary 2.3]{ET18}.

(1)$\Longrightarrow$(4)  The equality ${\rm G.gldim}(RG) = {\rm G.gldim}(R)$ might be well-known; see e.g. \cite[Example 3.8]{CR21}.  Moreover, for left Gorenstein regular rings $R$ and $RG$, the desired equalities hold immediately by \cite[Theorem 4.1]{Emm12}.

(4)$\Longrightarrow$(5) and (5)$\Longrightarrow$(6) are trivial, and (6)$\Longrightarrow$(1) follows by \cite[Theorem 3]{DT08}.
\end{proof}

Let $G$ be a group, $H$ be a finite subgroup of $G$, and denote by $N_{G}(H)$ the normalizer of $H$ in $G$. The Weyl group of $H$ is $W = N_{G}(H)/H$. Note that any module over the group $G$ on which the subgroup $H$ acts trivially induces a module over the Weyl group $W$.

We conclude this section by stating a result similar to \cite[Proposition 2.5]{ET18}. Due to Lemma \ref{lem:SplitMonic}, we can replace the rings of finite weak global dimension therein, to be Gorenstein regular rings. The argument follows verbatim from that of \cite{ET18}, so it is omitted. Note that we reobtain \cite[Theorem 2.8 (3)]{BDT09} by specifying $R$ to be the ring of integers $\mathbb{Z}$.

\begin{proposition}\label{prop:Weyl}
Let $R$ be a commutative Gorenstein regular ring, $G$ be a group. For any finite subgroup $H$ of $G$, we have ${\rm Gcd}_{R}(N_{G}(H)/H)\leq {\rm Gcd}_{R}G$.
\end{proposition}

\section{Extension of coefficient rings}

First, we consider the assignment ``Gcd'' when the coefficient rings are changed and the group is fixed.

\begin{proposition}\label{prop:RingOrd}
Let $R$ and $S$ be commutative rings, $G$ be any group. If $R$ is Gorenstein regular and $(G, S) \leq (G, R)$, then
$\mathrm{Gcd}_{S}G\leq \mathrm{Gcd}_{R}G$.
\end{proposition}

\begin{proof}
Since the inequality obviously holds if $\mathrm{Gcd}_{R}G = \infty$, it only suffices to consider the case where $\mathrm{Gcd}_{R}G = n$ is finite. In this case, it follows immediately from Lemma \ref{lem:SplitMonic} that there exists an $R$-split monomorphism of $RG$-modules $\iota: R\rightarrow \Lambda$, where $\Lambda$ is $R$-projective and $\mathrm{pd}_{RG}\Lambda = n$. Note that $S$ is an extension ring of $R$, then $S$ is both a trivial $SG$-module and an $R$-module. By applying $S\otimes_{R}-$, we have an $S$-split monomorphism of $SG$-modules $S\otimes_{R}\iota: S\rightarrow S\otimes_{R}\Lambda$, where $S\otimes_{R}\Lambda$ is $S$-projective since $\Lambda$ is $R$-projective.

Now we assume that $0\rightarrow P_n\rightarrow \cdots\rightarrow P_1\rightarrow P_0\rightarrow \Lambda\rightarrow 0$ is an $RG$-projective resolution of $\Lambda$. Note that the sequence is $R$-split. Then, by applying $S\otimes_{R}-$, we have an exact sequence of $SG$-modules
$$0\longrightarrow S\otimes_{R}P_n\longrightarrow \cdots\longrightarrow S\otimes_{R}P_1\longrightarrow S\otimes_{R}P_0\longrightarrow S\otimes_{R}\Lambda\longrightarrow 0.$$
Since $S\otimes_{R}P_i$ are projective $SG$-modules, we have $\mathrm{pd}_{SG}(S\otimes_{R}\Lambda)\leq n$. For the trivial $SG$-module $S$, it follows immediately from \cite[Proposition 1.4]{ET18} that
$\mathrm{Gpd}_{SG}S \leq \mathrm{pd}_{SG}(S\otimes_{R}\Lambda)$. Then, $\mathrm{Gcd}_{S}G \leq \mathrm{Gcd}_{R}G$ holds as required.
\end{proof}

Specifically, considering the pairs $(G, R) \leq (G, \mathbb{Z})$ and $(G, \mathbb{Q}) \leq (G, \mathbb{Z})$, we reobtain
\cite[Propostion 2.1]{ET18} and \cite[Theorem 3.2]{Tal14} respectively.

\begin{corollary}\label{cor:GcdZG}
Let $G$ be any group, $R$ be any commutative ring. Then $\mathrm{Gcd}_{R}G\leq \mathrm{Gcd}_{\mathbb{Z}}G$; in particular, $\mathrm{Gcd}_{\mathbb{Q}}G \leq \mathrm{Gcd}_{\mathbb{Z}}G$.
\end{corollary}

We are now in a position to show that ``Gcd'' preserves the order of pairs of groups and rings.

\begin{corollary}\label{cor:KpOrd}
Let $G$ and $H$ be groups, $R$ and $S$ be commutative Gorenstein regular rings. If $(H, S) \leq (G, R)$, then we have
$\mathrm{Gcd}_{S}H\leq \mathrm{Gcd}_{R}G$.
\end{corollary}

\begin{proof}
If $(H, S) \leq (G, R)$, then we have both $(H, S) \leq (G, S)$ and $(G, S) \leq (G, R)$. By Proposition \ref{prop:GroupOrd}, it follows that
 ${\rm Gcd}_{S}H \leq {\rm Gcd}_{S}G$. By Proposition \ref{prop:RingOrd}, $\mathrm{Gcd}_{S}G\leq \mathrm{Gcd}_{R}G$ holds. Then, we deduce the required inequality
$\mathrm{Gcd}_{S}H\leq \mathrm{Gcd}_{R}G$.
\end{proof}

Next, we give an upper bound for Gorenstein projective dimension of modules over group rings, by using Gorenstein cohomological dimension of the group.

\begin{proposition}\label{prop:gd-bound}
Let $G$ be a group, $R$ be a commutative Gorenstein regular ring. Then
$${\rm G.gldim}(RG) \leq \mathrm{Gcd}_{R}G + {\rm G.gldim}(R) \leq \mathrm{Gcd}_{\mathbb{Z}}G + {\rm G.gldim}(R).$$
\end{proposition}

\begin{proof}
Note that the second inequality follows from Corollary \ref{cor:GcdZG}. The first inequality obviously holds if ${\rm Gcd}_{R}G $ is infinite. Furthermore, if we assume the finiteness of ${\rm Gcd}_{R}G$, then the inequality follows immediately from Lemmas \ref{lem:SplitMonic} and \ref{lem:RG-Gp}.
\end{proof}

\begin{remark}
By the above result, we get \cite[Corollary 1.6]{ET18} immediately: if $R$ has finite global dimension and $G$ is a group with ${\rm Gcd}_RG < \infty$, then all $RG$-modules $M$ have finite Gorenstein projective dimension and ${\rm Gpd}_{RG}M \leq \mathrm{Gcd}_{R}G + {\rm gldim}(R)$.
\end{remark}

\begin{example}\label{eg:12}
Let $R$ be a commutative 1-Gorenstein ring. If $G$ is a finite group, then Gorenstein projective dimension of any $RG$-module is less than or equal to $1$. Recall from  \cite[Theorem 3.6]{BDT09} that the group $G$ acts on a tree with finite vertex stabilizers if and only if ${\rm Gcd}_{\mathbb{Z}}G \leq 1$. For example, any non-trivial free group $G$ satisfying ${\rm Gcd}_{\mathbb{Z}}G = {\rm cd}_{\mathbb{Z}}G = 1$. If $G$ is a group which acts on a tree with finite vertex stabilizers, then Gorenstein projective dimension of any $RG$-module is less than or equal to $2$.
\end{example}

\section{Cofibrant modules and Gorenstein projective modules}

Inspired by \cite{CK97, DT10}, in this section we intend to study the relation between Benson's cofibrant modules and Gorenstein projective modules over group rings.

Let $G$ be a group. The ring of bounded functions from $G$ to $\mathbb{Z}$ is denoted by $B(G, \mathbb{Z})$. For any commutative ring $R$, $B(G, R):= B(G, \mathbb{Z})\otimes_{\mathbb{Z}} R$ is the module of functions from $G$ to $R$ which take finitely many values; see \cite[Definition 3.1]{CK97}.

The $RG$-module structure on $B(G, R)$ is given by $g\alpha(g') = \alpha(g^{-1}g')$ for any $\alpha\in B(G, R)$ and any $g, g'\in G$. It follows from \cite[Lemma 3.4]{Ben97} and \cite[Lemma 3.2]{CK97} that $B(G, R)$ is free as an $R$-module, and restricts to a free $RF$-module for any finite subgroup $F$ of $G$. Let $M$ be an $RG$-module. Then $M\otimes_{R}B(G, R)$ is an $RG$-module, where $G$ acts diagonally on the tensor product.

The following notion is due to Benson; see \cite[Definition 4.1]{Ben97}.

\begin{definition}\label{def:Cof}
Let $M$ be an $RG$-module. Then $M$ is {\em cofibrant} if $M\otimes_{R}B(G, R)$ is a projective $RG$-module.
\end{definition}

\begin{proposition}\label{prop:Cof-GP}
Let $G$ be a group, $R$ a commutative ring of coefficients. Then any cofibrant $RG$-module is Gorenstein projective. Moreover, for any $RG$-module $M$, we have $\mathrm{Gpd}_{RG}M\leq \mathrm{pd}_{RG}M\otimes_{R}B(G, R)$.
\end{proposition}

\begin{proof}
The first assertion is essentially from \cite[Theorem 3.5]{CK97}; see also \cite[Proposition 4.1]{BDT09}. It remains to prove the inequality  $\mathrm{Gpd}_{RG}M\leq \mathrm{pd}_{RG}M\otimes_{R}B(G, R)$.

There is nothing to do unless we assume that  $\mathrm{pd}_{RG}M\otimes_{R}B(G, R) = n$ is finite. Consider a projective resolution $\cdots\rightarrow P_1\rightarrow P_0\rightarrow M\rightarrow 0$ of $M$, and let $K_n = \mathrm{Ker}(P_{n-1}\rightarrow P_{n-2})$. Since $P_i\otimes_{R}B(G, R)$ is projective for every $i\geq 0$, it follows from $\mathrm{pd}_{RG}M\otimes_{R}B(G, R) = n$ that $K_n\otimes_{R}B(G, R)$ is a projective $RG$-module, i.e. $K_n$ is cofibrant. Then $K_n$ is a Gorenstein projective $RG$-module, and
$\mathrm{Gpd}_{RG}M\leq  n$.
\end{proof}

Specifying the above $RG$-module $M$ to be the trivial $RG$-module $R$, we have $\mathrm{Gcd}_{R}G\leq \mathrm{pd}_{RG}B(G, R)$ immediately. Moreover, we have the following equality, which generalizes \cite[Theorem 1.18]{Bis21} from the ring of finite global dimension to Gorenstein regular ring.

\begin{corollary}\label{cor:Gcd=pdB1}
Let $R$ be a commutative Gorenstein regular ring, $G$ be any group. Then ${\rm Gcd}_{R}G = {\rm pd}_{RG}B(G, R)$ if the latter is finite.
\end{corollary}

\begin{proof}
By \cite[Lemma 3.3]{CK97}, there exists an $R$-split exact sequence of $RG$-modules
$0\rightarrow R\rightarrow B(G, R)$, where $B(G, R)$ is free as an $R$-module. If ${\rm pd}_{RG}B(G, R)$ is finite, we infer from Lemmas \ref{lem:SplitMonic} and \ref{lem:RG-Gp} that
${\rm pd}_{RG}B(G, R)= {\rm Gcd}_{R}G$ holds.
\end{proof}

\begin{lemma}\label{lem:cof}
Let $M$ be a Gorenstein projective $RG$-module. Then $M$ is cofibrant $($i.e. ${\rm pd}_{RG}M\otimes_{R}B(G, R) = 0)$ if and only if ${\rm pd}_{RG}M\otimes_{R}B(G, R) < \infty$.
\end{lemma}

\begin{proof}
The ``only if'' part is trivial. Assume that ${\rm pd}_{RG}M\otimes_{R}B(G, R)$ is finite. It follows from \cite[Lemma 4.5(ii)]{Ben97} that if every extension of $M$ by an $RG$-module of finite projective dimension splits, then $M$ is cofibrant. It is a well-known fact that for any Gorenstein projective module $M$ and any $RG$-module $W$ of finite projective dimension, one has ${\rm Ext}^i_{RG}(M, W) = 0$ for any $i\geq 1$. Then, the ``if'' part holds.
\end{proof}

We now have some basic properties of cofibrant modules. We use $\mathcal{C}of(RG)$ to denote the class of cofibrant $RG$-modules.

\begin{proposition}\label{prop:cof}
Let $G$ be a group and $R$ be a commutative ring of coefficients. Then $\mathcal{C}of(RG)$ contains all projective $RG$-modules. For any exact sequence  $0\rightarrow L\rightarrow M\rightarrow N\rightarrow 0$ of $RG$-modules, if $N$ is cofibrant, then $M$ is cofibrant if and only if so is $L$. If $L$ and $M$ are cofibrant and ${\rm Ext}^1_{RG}(N, Q) = 0$ for any $RG$-module $Q$ of finite projective dimension, then $N$ is cofibrant as well.
\end{proposition}

\begin{proof}
It is sufficient to prove the last assertion, since the others follow by the definition of cofibrant modules. For any exact sequence $0\rightarrow L\rightarrow M\rightarrow N\rightarrow 0$, where $L$ and $M$ are cofibrant, by applying the functor $-\otimes_{R}B(G, R)$ we obtain that ${\rm pd}_{RG}N\otimes_{R}B(G, R) \leq 1$. Furthermore, if ${\rm Ext}^1_{RG}(N, Q) = 0$ for any $RG$-module $Q$ with ${\rm pd}_{RG}Q < \infty$, then it follows immediately from  \cite[Lemma 4.5(ii)]{Ben97} that $N$ is cofibrant.
\end{proof}

It was suggested in \cite[Conjecture A]{DT10} that over the integral group ring $\mathbb{Z}G$, Gorenstein projective modules coincide with Benson's cofibrant modules. For the relations between these two classes of modules, we have the following.

Recall that a group $G$ is of type $\Phi_R$, if it has the property that an $RG$-module is of finite projective dimension if and only if its restriction to any finite subgroup is of finite projective dimension; see \cite[Definition 2.1]{MS19}. It was conjectured that a group $G$ is of type $\Phi_\mathbb{Z}$ if and only if $G$ admits a finite dimensional model for $\underline{E}G$, the classifying space for the family of the finite subgroups of $G$; the sufficient condition was proved to be true. We refer to \cite{Tal07} for the details.

\begin{proposition}\label{prop:GP-Cof1}
Let $R$ be a commutative ring, and $G$ be a group of type $\Phi_R$. For any Gorenstein projective $RG$-module $M$, if ${\rm pd}_RM < \infty$, then $M$ is cofibrant.
\end{proposition}

\begin{proof}
It follows from \cite[Lemma 3.4(ii)]{Ben97} that for any finite subgroup $H$ of $G$, $B(G, R)$ is restricted to be a free $RH$-module. For any $RG$-module $M$, if ${\rm pd}_RM$ is finite, then ${\rm pd}_{RH}M\otimes_{R}B(G, R)$ is also finite. Since $G$ is of type $\Phi_R$,  we infer that ${\rm pd}_{RG}M\otimes_{R}B(G, R)$ is finite. Moreover, if $M$ is Gorenstein projective, then it follows directly from Lemma \ref{lem:cof} that $M$ is cofibrant.
\end{proof}

Kropholler introduced in \cite{Kro93} the notion of $\mathrm{LH}\mathfrak{F}$-groups. Let $\mathfrak{F}$ be the class of finite groups. A class of groups
${\rm H}_{\alpha}\mathfrak{F}$ for each ordinal $\alpha$ is defined as follows. Let ${\rm H}_0\mathfrak{F} = \mathfrak{F}$, and for $\alpha > 0$ we define a group $G$ to be in
${\rm H}_{\alpha}\mathfrak{F}$ if $G$ acts cellularly on a finite dimensional contractible CW-complex $X$, in such a way that the setwise stabilizer of each cell is equal to the pointed stabilizer, and is in ${\rm H}_{\beta}\mathfrak{F}$ for some $\beta < \alpha$. The class ${\rm H}\mathfrak{F}$ is then the union of all the ${\rm H}_{\alpha}\mathfrak{F}$. Here, the letter ``H'' stands for ``hierarchy''. Finally, a group $G$ is said to be an ${\rm LH}\mathfrak{F}$-group if every finitely generated subgroup of $G$ is an ${\rm H}\mathfrak{F}$-group. It follows from \cite{Kro93, BDT09} that soluble-by-finite groups, linear groups, groups of finite cohomological dimension over $\mathbb{Z}$, and groups with a faithful representation as endomorphisms of a noetherian module over a commutative ring, are all contained in $\mathrm{LH}\mathfrak{F}$.

Inspired by \cite[Theorem B]{DT10}, we have the following.

\begin{proposition}\label{prop:GP-Cof2}
Let $G$ be an $\mathrm{LH}\mathfrak{F}$-group, $R$ a commutative Gorenstein regular ring. For any Gorenstein projective $RG$-module $M$, if ${\rm pd}_RM < \infty$, then $M$ is cofibrant.
\end{proposition}

\begin{proof}
We first prove the assertion for $\mathrm{H}\mathfrak{F}$-groups, using transfinite induction on the ordinal number $\alpha$ such that $G$ belongs to $\mathrm{H}_{\alpha}\mathfrak{F}$. For $\alpha = 0$, i.e. $G$ is a finite group, $B(G, R)$ is a free $RG$-module, and then the finiteness of ${\rm pd}_{RG}M\otimes_{R}B(G, R)$ follows from the hypothesis ${\rm pd}_RM < \infty$; furthermore, $M$ is cofibrant by Lemma \ref{lem:cof}. Now let us assume that the result is true for any $H\in \mathrm{H}_{\beta}\mathfrak{F}$ and all $\beta < \alpha$. A crucial algebraic consequence of the definition of $\mathrm{H}\mathfrak{F}$-groups is that there exists an augmented cellular chain on a finite dimensional contractible CW-complex $X$, that is, we have an exact sequence of $\mathbb{Z}G$-modules
$$0\longrightarrow C_n\longrightarrow \cdots\longrightarrow C_1\longrightarrow C_0\longrightarrow \mathbb{Z}\longrightarrow 0,$$
where each $C_i$ is a direct sum of permutation modules of the form $\mathbb{Z}[G/H]$ with $H$ a subgroup of $G$ satisfying $H\in \mathrm{H}_{\beta}\mathfrak{F}$ for $\beta < \alpha$.

For any subgroup $H$ of $G$, $RG$ is a free right $RH$-module since the right translation action of $H$ on $G$ is free, and we can take any set of representatives for the left cosets $gH$ as a basis. Then, for any projective $RH$-module $P$, it follows from the adjunction ${\rm Hom}_{RG}(RG\otimes_{RH}P, -)\cong {\rm Hom}_{RH}(P, {\rm Hom}_{RG}(RG, -))$ that the module ${\rm Ind}_H^GP = RG\otimes_{RH}P$ is projective over $RG$. For each $\mathbb{Z}[G/H] = \mathrm{Ind}_{H}^{G}\mathbb{Z}$, there is an isomorphism (see \cite[Proposition III 5.6]{Bro82})
$${\rm Ind}_H^G\mathbb{Z}\otimes_{\mathbb{Z}}(M\otimes_{R}B(G, R)) \cong {\rm Ind}_H^G(M\otimes_{R}B(G, R)).$$
Invoking the induction hypothesis that $M\otimes_{R}B(G, R)$ is a projective $RH$-module, it follows that $\mathbb{Z}[G/H]\otimes_{\mathbb{Z}}(M\otimes_{R}B(G, R))$ is projective over $RG$ and thus each $C_i\otimes_{\mathbb{Z}}(M\otimes_{R}B(G, R))$ is projective over $RG$. Tensoring the above sequence with $M\otimes_{R}B(G, R)$ gives us an exact sequence of $RG$-modules
$$0\rightarrow C_n\otimes_{\mathbb{Z}}(M\otimes_{R}B(G, R))\rightarrow \cdots\rightarrow C_0\otimes_{\mathbb{Z}}(M\otimes_{R}B(G, R))\rightarrow M\otimes_{R}B(G, R)\rightarrow 0,$$
which implies that ${\rm pd}_{RG}M\otimes_{R}B(G, R) < \infty$ for any Gorenstein projective $RG$-module $M$. Hence, we infer from Lemma \ref{lem:cof} that $M$ is cofibrant.

Next, we prove the result for $\mathrm{LH}\mathfrak{F}$-groups by induction on the cardinality of the group. If $G$ is a countable group, then it belongs to $\mathrm{H}\mathfrak{F}$, so the assertion holds by the above argument. Assume that $G$ is uncountable and the result is true for any groups with cardinality strictly smaller than that of $G$. Then $G$ can be expressed as the union of an ascending chain of subgroups $G_{\alpha}$, where $\alpha < \gamma$ for some ordinal number $\gamma$, such that each $G_{\alpha}$ has cardinality strictly smaller that $|G|$. By the induction hypothesis, $M\otimes_{R}B(G, R)$ is projective over each $RG_{\alpha}$. Then, it follows directly from \cite[Lemma 5.6]{Ben97} that ${\rm pd}_{RG}M\otimes_{R}B(G, R) \leq 1$. By Lemma \ref{lem:cof}, it follows that the Gorenstein projective $RG$-module $M$ is cofibrant.
\end{proof}

We might get the following, which recovers \cite[Corollary C]{DT10}, \cite[Theorem 3.11]{Bis21+} and \cite[Theorem 3.4]{Bis21}.

\begin{corollary}\label{cor:GP=Cof}
Let $R$ be a commutative ring of finite global dimension. If $G$ is either a group of type $\Phi_R$, or an $\mathrm{LH}\mathfrak{F}$-group, then Gorenstein projective $RG$-modules coincide with cofibrant modules. Moreover, for any $RG$-module $M$, we have  $\mathrm{pd}_{RG}M\otimes_{R}B(G, R) = \mathrm{Gpd}_{RG}M$.
\end{corollary}

\begin{proof}
Invoking Proposition \ref{prop:Cof-GP}, it suffices to prove that for any $RG$-module $M$,
$\mathrm{pd}_{RG}M\otimes_{R}B(G, R) \leq \mathrm{Gpd}_{RG}M$. Since the inequality which is to be proved is obvious if
$\mathrm{Gpd}_{RG}M = \infty$, it only suffices to consider the case where $\mathrm{Gpd}_{RG}M = n$ is finite. Then,
there exists an exact sequence of $RG$-modules
$$0\longrightarrow M_n\longrightarrow \cdots\longrightarrow M_1\longrightarrow M_0\longrightarrow M\longrightarrow 0,$$
where $M_0, M_1, \cdots, M_n$ are Gorenstein projective; they are also cofibrant modules by
Propositions \ref{prop:GP-Cof1} and \ref{prop:GP-Cof2}. By applying $-\otimes_{R}B(G, R)$ to this sequence,
we infer that $\mathrm{pd}_{RG}M\otimes_{R}B(G, R) \leq n$ since $M_i\otimes_{R}B(G, R)$ are projective $RG$-modules.
 \end{proof}

In particular, for the trivial $RG$-module $R$ we have the following, which supports a question raised in \cite[Conjecture 3.4]{Bis21+} and gives an example to Theorem \ref{thm:fGcd}.

\begin{corollary}\label{cor:Gcd=pdB2}
$\mathrm{Gcd}_RG = \mathrm{pd}_{RG}B(G, R)< \infty$ if either of the following holds:
\begin{enumerate}
\item $R$ is a commutative Gorenstein regular ring and $G$ is a group of type $\Phi_R$;
\item $R$ is a commutative ring of finite global dimension and $G$ is an $\mathrm{LH}\mathfrak{F}$-group of type $FP_\infty$.
\end{enumerate}
\end{corollary}

\begin{proof}
If $G$ is of type $\Phi_R$, then by \cite[Lemma 3.4(ii)]{Ben97} we have $\mathrm{pd}_{RG}B(G, R)< \infty$, and the assertion is true by Corollary \ref{cor:Gcd=pdB1}. If $G$ is an $\mathrm{LH}\mathfrak{F}$-group of type $FP_\infty$, then by \cite[Theorem A.2]{ET22} we have $\mathrm{Gcd}_{\mathbb{Z}}G < \infty$. Moreover, by Corollaries \ref{cor:GcdZG} and \ref{cor:GP=Cof}
we have $\mathrm{pd}_{RG}B(G, R) = \mathrm{Gcd}_RG \leq \mathrm{Gcd}_{\mathbb{Z}}G < \infty$.
\end{proof}

\section{Model structure and Stable categories over group rings}

A model category \cite{Qui67} refers to a category with three specified classes of morphisms, called fibrations, cofibrations and weak equivalences, which satisfy a few axioms that are deliberately reminiscent of properties of topological spaces. The homotopy category associated to a model category is obtained by formally inverting the weak equivalences, while the objects are the same. We refer to \cite{Qui67, Hov99} for basic definitions and facts on model categories.

\begin{lemma}\label{lem:C}
Let $G$ be a group and $R$ be a commutative ring of coefficients. Then the subcategory of cofibrant $RG$-modules $\mathcal{C}of(RG)$ is a Frobenius category, whose projective and injective objects coincide and are precisely the projective $RG$-modules.
\end{lemma}

\begin{proof}
By Proposition \ref{prop:Cof-GP}, for any projective $RG$-module $P$ and any cofibrant module $M$, we have ${\rm Ext}_{RG}^i(M, P) = 0$, and then projective $RG$-modules are the projective-injective objects with respect to $\mathcal{C}of(RG)$. Consequently, by combing with the properties of cofibrant modules in Proposition \ref{prop:cof}, we get that $\mathcal{C}of(RG)$ is a Frobenius category.
\end{proof}

Let $A$ be a ring. For any $A$-modules $M$ and $N$, we define $\underline{{\rm Hom}}_A(M, N)$ to be the quotient of ${\rm Hom}_A(M, N)$ by the additive subgroup consisting of homomorphisms which factor through a projective module. For the Frobenius category $\mathcal{C}of(RG)$ of cofibrant $RG$-modules, the stable category
$\underline{\mathcal{C}of}(RG)$ is a triangulated category with objects being cofibrant modules but morphisms being $\underline{{\rm Hom}}_{RG}(M, N)$. %We refer to \cite{Hap88} for details of Frobenius categories and stable categories.

We refer to \cite{Bun10} for the notion of exact category, which is commonly attributed to Quillen. By \cite[Definition 2.2]{Gil11}, an additive category is weakly idempotent complete if every split monomorphism has a cokernel and every split epimorphism has a kernel. The following is immediate.

\begin{lemma}\label{lem:F}
Let $G$ be a group and $R$ be a commutative ring of coefficients. Let $\mathcal{F}ib(RG)$ be the subcategory of $RG$-modules $M$ such that ${\rm pd}_{RG}M\otimes_{R} B(G, R)< \infty$. Then $\mathcal{F}ib(RG)$ is a weakly idempotent complete exact category. The modules in $\mathcal{F}ib(RG)$ are called fibrant $RG$-modules.
\end{lemma}

Let $\mathcal{A}$ be an exact category with enough projectives and enough injectives. The Yoneda Ext bifunctor is denoted by $\mathrm{Ext}_{\mathcal{A}}(-, -)$. A pair of classes $(\mathcal{X}, \mathcal{Y})$ in $\mathcal{A}$ is a \emph{cotorsion pair} provided that $\mathcal{X} =  {^\perp}\mathcal{Y}$ and $\mathcal{Y} = \mathcal{X}^{\perp}$, where the left orthogonal class $^{\perp}\mathcal{Y}$ consists of $X$ such that $\mathrm{Ext}^{\geq 1}_{\mathcal{A}}(X, Y) = 0$ for all $Y\in \mathcal{Y}$, and the right orthogonal class $\mathcal{X}^{\perp}$ is defined similarly. The cotorsion pair $(\mathcal{X}, \mathcal{Y})$ is said to be \emph{complete} if for any object $M\in \mathcal{A}$, there exist short exact sequences $0\rightarrow Y\rightarrow X \rightarrow M \rightarrow 0$ and $0\rightarrow M\rightarrow Y' \rightarrow X' \rightarrow 0$ with $X, X'\in \mathcal{X}$ and $Y, Y'\in \mathcal{Y}$.

It is not usually trivial to check that a category has a model category structure, see for example \cite[Section 10]{Ben97}. However, thanks to \cite[Theorem 2.2]{Hov02} and \cite[Theorem 3.3]{Gil11}, we have a correspondence between complete cotorsion pairs and the model structure. This brings us a convenience to have the following.

\begin{theorem}\label{thm:model}
Let $G$ be a group and $R$ be a commutative ring of coefficients. Let $\mathcal{W}$ be the subcategory formed by $RG$-modules of finite projective dimension. Then there are complete cotorsion pairs
$(\mathcal{C}of\cap \mathcal{W}, \mathcal{F}ib)$ and $(\mathcal{C}of, \mathcal{W}\cap \mathcal{F}ib)$ in the weakly idempotent complete exact category $\mathcal{F}ib(RG)$.

Consequently, there is a model structure on the category $\mathcal{F}ib(RG):$
\begin{itemize}
\item the cofibrations  (trivial cofibrations) are monomorphisms whose cokernels are cofibrant $RG$-modules (projective $RG$-modules).
\item the fibrations (trivial fibrations) are epimorphisms (with kernel being of finite projective dimension).
\item the weak equivalences are morphisms which factor as a trivial cofibration followed by a trivial fibration.
\end{itemize}
\end{theorem}

\begin{proof}
It is clear that $\mathcal{P}\subseteq \mathcal{C}of\cap \mathcal{W}$, that is, all projective $RG$-modules are included in  $\mathcal{C}of\cap \mathcal{W}$. We infer that $\mathcal{C}of\cap \mathcal{W}\subseteq \mathcal{P}$ since any cofibrant module is Gorenstein projective, and projective dimension of any Gorenstein projective module is either zero or infinity. Hence, $\mathcal{C}of\cap \mathcal{W} = \mathcal{P}$.

For any $P\in \mathcal{P}$ and any $RG$-module $M$, it is clear that ${\rm Ext}^{\geq 1}_{RG}(P, M) = 0$, and furthermore, we have $\mathcal{P}^{\perp} = \mathcal{F}ib$ and $\mathcal{P} \subseteq {^{\perp}\mathcal{F}ib}$ by noting that ``$\perp$'' is only calculated inside $\mathcal{F}ib$. Let $M\in {^{\perp}\mathcal{F}ib}$. There is an exact sequence $0\rightarrow K\rightarrow P\rightarrow M\rightarrow 0$ in $\mathcal{F}ib$, where $P$ is a projective $RG$-module. Noting that
${\rm Ext}^1_{RG}(M, K) = 0$, we deduce that the sequence is split, and hence, as a direct summand of $P$, $M$ is projective. This implies the inclusion ${^{\perp}\mathcal{F}ib}\subseteq \mathcal{P}$,  and consequently, we obtain a cotorsion pair $(\mathcal{C}of \cap \mathcal{W}, \mathcal{F}ib) = (\mathcal{P}, \mathcal{F}ib)$. The completeness of this cotorsion pair is easy to see.

Next, we show that $(\mathcal{C}of, \mathcal{W}\cap \mathcal{F}ib) = (\mathcal{C}of, \mathcal{W})$ is a cotorsion pair. Since every cofibrant module is Gorenstein projective,  $\mathcal{C}of \subseteq {^{\perp}\mathcal{W}}$ and $\mathcal{W}\subseteq \mathcal{C}of^{\perp}$ hold immediately. For any $M\in {^{\perp}\mathcal{W}}$, we have $M\in \mathcal{C}of$  by Lemma \ref{lem:cof} since we only consider objects in $\mathcal{F}ib$. Hence, ${^{\perp}\mathcal{W}} \subseteq \mathcal{C}of$, and then
$\mathcal{C}of = {^{\perp}\mathcal{W}}$. Let $M$ be any object in $\mathcal{C}of^{\perp}$. Since we also have $M\in \mathcal{F}ib$, it follows from
Proposition \ref{prop:Cof-GP} that ${\rm Gpd}_{RG}M \leq {\rm pd}_{RG}M\otimes_{R}B(G, R)$ is finite. Assume ${\rm Gpd}_{RG}M = n$. By an argument analogous to that of Lemma \ref{lem:SplitMonic}, we have an exact sequence $0\rightarrow M\rightarrow N\rightarrow L\rightarrow 0$ from a pushout diagram, where $L$ is Gorenstein projective and
${\rm pd}_{RG}N = n$. Then, we infer from Lemma \ref{lem:cof} that $L$ is cofibrant by noting $L\in \mathcal{F}ib$. Hence, ${\rm Ext}_{RG}^1(L, M) = 0$ for  $M\in \mathcal{C}of^{\perp}$. Then, the above sequence is split, and ${\rm pd}_{RG}M\leq {\rm pd}_{RG}N = n$. This implies that $\mathcal{C}of^{\perp} \subseteq \mathcal{W}$, and finally, $(\mathcal{C}of, \mathcal{W})$ is a cotorsion pair.

Let $M$ be any $RG$-module in $\mathcal{F}ib$. By the above argument we have an exact sequence $0\rightarrow M\rightarrow N\rightarrow L\rightarrow 0$ with $N\in \mathcal{W}$ and $L\in \mathcal{C}of$. Moreover, since ${\rm Gpd}_{RG}M \leq {\rm pd}_{RG}M\otimes_{R}B(G, R)$ is finite, it follows from \cite[Theorem 2.10]{Hol04} that there exists an exact sequence
$0\rightarrow K\rightarrow H\rightarrow M\rightarrow 0$ in $\mathcal{F}ib$  where $K\in \mathcal{W}$ and $H$ is a Gorenstein projective $RG$-module. By Lemma \ref{lem:cof}, $H$ is also a cofibrant module. Hence, the completeness of the cotorsion pair
$(\mathcal{C}of, \mathcal{W})$ follows.

Consequently, by using \cite[Theorem 3.3]{Gil11} we have a model structure on $\mathcal{F}ib$ as stated above, which is corresponding to the triple of classes of $RG$-modules $(\mathcal{C}of, \mathcal{W}, \mathcal{F}ib)$. The triple is usually referred to as a Hovey triple, since such a correspondence was obtained by Hovey in \cite[Theorem 2.2]{Hov02}.
\end{proof}

Let $M$ be an object in a model category. Recall that $M$ is called {\em cofibrant} if $0\rightarrow M$ is a cofibration, and it is called {\em fibrant} if $M\rightarrow 0$ is a fibration. For the Frobenius category $\mathcal{C}of$, whose objects are both cofibrant and fibrant, there is a stable category. Moreover, for the model category $\mathcal{F}ib$ there is an associated homotopy category $\mathrm{Ho}(\mathcal{F}ib)$, which is obtained by formally inverting weak equivalences, that is, the localization of $\mathcal{F}ib$ with respect to the class of weak equivalences.

The following is immediate from a fundamental result about model categories, see for example \cite[Theorem 1.2.10]{Hov99}.

\begin{corollary}\label{cor:equ1}
Let $G$ be a group and $R$ be a commutative ring of coefficients. The homotopy category ${\rm Ho}(\mathcal{F}ib)$ is triangle equivalent to the stable category $\underline{\mathcal{C}of}(RG)$.
\end{corollary}

Note that weak equivalences are crucial in the model category. We have the following characterization, which might be of independent interest.

\begin{proposition}\label{prop:WEqu}
Let $G$ be a group and $R$ a commutative ring of coefficients. For any cofibrant $RG$-modules $M$ and $N$, the following are equivalent:
\begin{enumerate}
\item There is a weak equivalence between $M$ and $N$;
\item $M$ and $N$ are isomorphic in the stable category $\underline{\mathcal{C}of}(RG)$;
\item There is an isomorphism $M\oplus P\cong N\oplus Q$ in the category ${\rm Mod}(RG)$ of $RG$-modules, for some projective modules $P$ and $Q$.
\end{enumerate}
\end{proposition}

\begin{proof}
By Corollary \ref{cor:equ1}, $(1)\Rightarrow (2)$ is immediate. In general, for any stable category of a Frobenius category, $(2)\Rightarrow (3)$ follows; see e.g. \cite[Lemma 1.1]{CZ07}.

$(3)\Longrightarrow (1)$ Note that the isomorphism $M\oplus P\rightarrow  N\oplus Q$ is a weak equivalence in the model category $\mathcal{F}ib$. Moreover, the injection $M\rightarrow M\oplus P$ is a trivial cofibration, and the projection $N\oplus Q\rightarrow N$ is a trivial fibration. By the concatenation of these maps, we get a weak equivalence between $M$ and $N$.
\end{proof}

For any module $M$, we define $\Omega(M)$ to be the kernel of a surjective homomorphism from a projective module onto $M$. For any $RG$-modules $M$ and $N$, there is a natural homomorphism $\underline{{\rm Hom}}_{RG}(M, N) \rightarrow \underline{{\rm Hom}}_{RG}(\Omega(M), \Omega(N))$. The complete cohomology is defined by
$\widehat{{\rm Ext}}_{RG}^n(M, N) = \mathop{\underrightarrow{\mathrm{lim}}}\limits_i\underline{{\rm Hom}}_{RG}(\Omega^{n+i}(M), \Omega^i(N))$.

We have the following definition, which generalizes the one in \cite[Section 8]{Ben97} by removing the countably presented condition for modules.

\begin{definition}\label{def:stab}
Let ${\rm StMod}(RG)$ be the stable module category with all fibrant $RG$-modules as objects, and morphisms for any objects $M$ and $N$ given by complete cohomology of degree zero, that is,
$${\rm Hom}_{{\rm StMod}(RG)}(M, N) = \mathop{\underrightarrow{\mathrm{lim}}}\limits_i\underline{{\rm Hom}}_{RG}(\Omega^i(M), \Omega^i(N)).$$
\end{definition}

Note that a map in ${\rm Hom}_{{\rm StMod}(RG)}(M, N)$ might not correspond to any map in ${\rm Hom}_{RG}(M, N)$, and hence, this definition can be difficult to work with. However, for cofibrant modules, we have the following.

\begin{lemma}\label{lem:hom}
Let $M$ and $N$ be cofibrant $RG$-modules. There is an isomorphism $$\underline{{\rm Hom}}_{RG}(M, N) \cong \underline{{\rm Hom}}_{RG}(\Omega(M), \Omega(N)).$$
\end{lemma}

\begin{proof}
There exist exact sequences $0\rightarrow \Omega(M)\rightarrow P\rightarrow M\rightarrow 0$ and  $0\rightarrow \Omega(N)\rightarrow Q\rightarrow N\rightarrow 0$, where $P$ and $Q$ are projective $RG$-modules. Since ${\rm Ext}_{RG}^1(M, Q) = 0$, any map from $\Omega(M)$ to $Q$ can be extended to a map $P\rightarrow Q$, and then we have the following commutative diagram
$$\xymatrix@C=20pt@R=20pt{
 0\ar[r] & \Omega(M)\ar[d]_{\Omega(f)}\ar[r] & P\ar[r]\ar[d]^{g} &M\ar[r]\ar[d]^{f} & 0 \\
0\ar[r] & \Omega(N) \ar[r] &Q \ar[r] &N\ar[r] & 0
  }$$
Hence, the natural homomorphism $\underline{{\rm Hom}}_{RG}(M, N) \rightarrow \underline{{\rm Hom}}_{RG}(\Omega(M), \Omega(N))$ is surjective.

Now assume that $0 = \Omega(f)\in \underline{{\rm Hom}}_{RG}(\Omega(M), \Omega(N))$. Since ${\rm Ext}_{RG}^1(M, -)$ vanishes for any projective $RG$-module, if $\Omega(f):\Omega(M)\rightarrow \Omega(N)$ factors through a projective module, then it might factor through $P$. That is, for $\alpha: \Omega(M)\rightarrow P$, there exists a map $s: P\rightarrow \Omega(N)$ such that $\Omega(f) = s\alpha$. Denote by $\gamma: \Omega(N)\rightarrow Q$ and $\delta: Q\rightarrow N$. It is standard to show that $g - \gamma s: P\rightarrow Q$ factors through $\beta: P\rightarrow M$, that is, there exists a map $t: M\rightarrow Q$ such that $g - \gamma s = t\beta$. Since
$f\beta =  \delta g = \delta t \beta$ and $\beta$ is epic, it yields that $f = \delta t$. Hence, the natural homomorphism  $\underline{{\rm Hom}}_{RG}(M, N) \rightarrow \underline{{\rm Hom}}_{RG}(\Omega(M), \Omega(N))$ is also injective.
\end{proof}

\begin{theorem}\label{thm:stable}
Let $G$ be a group and $R$ be a commutative ring of coefficients. Then ${\rm Ho}(\mathcal{F}ib)$ is equivalent to ${\rm StMod}(RG)$.
\end{theorem}

\begin{proof}
First, we note that objects of ${\rm Ho}(\mathcal{F}ib)$ and ${\rm StMod}(RG)$ coincide. It suffices to prove that the natural functor from ${\rm Ho}(\mathcal{F}ib)$ to
${\rm StMod}(RG)$ is fully faithful.

Let $M$ and $N$ be any fibrant $RG$-modules. By the completeness of the cotorsion pair $(\mathcal{C}of, \mathcal{W})$, there exists an exact sequence $0\rightarrow K_M\rightarrow Q(M)\rightarrow M\rightarrow 0$,
where $Q(M)$ is cofibrant and $K_M\in \mathcal{W}$. Hence, the cofibrant approximation $Q(M)\rightarrow M$ is also a trivial fibration, and we refer to it (or simply, to $Q(M)$) as being a cofibrant replacement of $M$. Then, we have the following isomorphisms
$${\rm Hom}_{{\rm Ho}(\mathcal{F}ib)}(M, N)\cong \underline{{\rm Hom}}_{RG}(Q(M), Q(N)) \cong {\rm Hom}_{{\rm StMod}(RG)}(Q(M), Q(N)),$$
where the first one follows by \cite[Theorem 1.2.10(ii)]{Hov99} and Proposition \ref{prop:WEqu}, and the second one holds by Lemma \ref{lem:hom}.

By basic properties of cofibrant modules (see Proposition \ref{prop:cof}), for fibrant $RG$-modules $M$ and $N$, there exists an integer $r >> 0$, such that both $\Omega^r(M)$ and $\Omega^r(N)$ are cofibrant modules, and moreover, the projective dimension of $K_M$ and $K_N$ are not more than $r - 1$. For $M$, we have the following commutative diagram with exact rows
$$\xymatrix@C=20pt@R=20pt{
0\ar[r] &\Omega^r(M) \ar[r]\ar[d] &P_{r-1}\ar[r]\ar[d] &\cdots \ar[r] &P_1\ar[r]\ar[d] &P_0\ar[r]\ar[d]
&M \ar[r] \ar@{=}[d] &0\\
0 \ar[r] &P'_{r}\ar[r]  &P'_{r-1}\ar[r] & \cdots \ar[r] &P'_{1}\ar[r] &Q(M) \ar[r] &M \ar[r] &0 }$$
where $P_i$ and $P'_i$ are all projective $RG$-modules; this yields an exact sequence
$$0\longrightarrow\Omega^r(M)\longrightarrow P_{r-1}\oplus P'_{r}\longrightarrow\cdots \longrightarrow
P_0\oplus P'_1 \longrightarrow Q(M)\longrightarrow 0.$$
Similarly, we obtain such an exact sequence for $Q(N)$. Moreover, since $P_0\oplus P'_1$, $\cdots$, $P_{r-1}\oplus P'_{r}$ are projective modules, we get the following commutative diagram:
$$\xymatrix@C=20pt@R=20pt{
 0\ar[r] & \Omega^r(M)\ar[d]_{\Omega^r(f)}\ar[r] & P_{r-1}\oplus P'_{r}\ar[r]\ar[d] &\cdots \ar[r] &P_0\oplus P'_1 \ar[r]\ar[d] & Q(M)\ar[r]\ar[d]^{f} & 0 \\
 0\ar[r] & \Omega^r(N)\ar[r] & Q_{r-1}\oplus Q'_{r}\ar[r] &\cdots \ar[r] &Q_0\oplus Q'_1 \ar[r] & Q(N)\ar[r] & 0
  }$$
Analogous to Lemma \ref{lem:hom}, we can prove that there is an isomorphism
$$\underline{{\rm Hom}}_{RG}(Q(M), Q(N)) \cong \underline{{\rm Hom}}_{RG}(\Omega^r(M), \Omega^r(N)).$$
Moreover, it follows from Lemma \ref{lem:hom} that for all $j > 0$, we have isomorphisms
$$\underline{{\rm Hom}}_{RG}(\Omega^r(M), \Omega^r(N)) \cong \underline{{\rm Hom}}_{RG}(\Omega^{r+j}(M), \Omega^{r+j}(N)),$$
and consequently,
$${\rm Hom}_{{\rm StMod}(RG)}(M, N) = \mathop{\underrightarrow{\mathrm{lim}}}\limits_i\underline{{\rm Hom}}_{RG}(\Omega^i(M), \Omega^i(N)) = \underline{{\rm Hom}}_{RG}(\Omega^r(M), \Omega^r(N)).$$
Hence, we get the desired isomorphism ${\rm Hom}_{{\rm Ho}(\mathcal{F}ib)}(M, N)\cong {\rm Hom}_{{\rm StMod}(RG)}(M, N)$. We are done with the proof.
\end{proof}

It is well known that the category of Gorenstein projective modules $\mathcal{GP}(A)$ over any ring $A$ is a Frobenius category, then the corresponding stable category $\underline{\mathcal{GP}}(A)$ is a triangulated category.
We denote by ${\rm D}^b(RG)$ the bounded derived category of $RG$, and let ${\rm K}^b(RG)$ and ${\rm K}^b(RG\text{-}{\rm Proj})$ stand for the homotopy categories of bounded $RG$-complexes and bounded complexes of projective $RG$-modules, respectively. Note that the composite of natural functors $${\rm K}^b(RG\text{-}{\rm Proj})\rightarrow {\rm K}^b(RG)\rightarrow {\rm D}^b(RG)$$ is fully faithful, and this allows us to view ${\rm K}^b(RG\text{-}{\rm Proj})$
as a (thick) triangulated subcategory of ${\rm D}^b(RG)$. The singularity category of $RG$, denoted by ${\rm D}_{sg}(RG)$, is defined to be the Verdier quotient
${\rm D}^b(RG)/ {\rm K}^b(RG\text{-}{\rm Proj})$; see \cite{Buc87, Or04} or \cite[Section 6]{Bel00}. Note that ${\rm D}_{sg}(RG)$ vanishes if and only if every module has finite projective dimension, that is,  $RG$ is a ring of finite global dimension.

Consider the following composition of functors
$$F: \mathcal{GP}(RG)\hookrightarrow {\rm Mod}(RG)\stackrel{\iota}\longrightarrow {\rm D}^b(RG) \stackrel{\pi}\longrightarrow {\rm D}_{sg}(RG),$$
where the first functor is the inclusion, the second functor $\iota$ is the full embedding which sends every $RG$-module to the corresponding stalk complex concentrated in degree zero, and the last one $\pi$ is the standard quotient functor. Note that $F$ induces a unique functor from the stable category $\underline{\mathcal{GP}}(RG)$ to ${\rm D}_{sg}(RG)$.

Let ${\rm K}_{ac}(RG\text{-}{\rm Proj})$ and ${\rm K}_{tac}(RG\text{-}{\rm Proj})$ denote the homotopy categories of acyclic complexes and totally acyclic complexes of projective $RG$-modules, respectively. There is a functor
$\Omega: {\rm K}_{tac}(RG\text{-}{\rm Proj})\rightarrow \underline{\mathcal{GP}}(RG)$ given by taking the kernel of the boundary map in degree zero.

We summarize some equivalences of triangulated categories as follows, which generalize \cite[Theorem 3.10]{MS19} where the equivalences
$${\rm StMod}(RG) \simeq \underline{\mathcal{GP}}(RG) \simeq {\rm D}_{sg}(RG) \simeq {\rm K}_{tac}(RG\text{-}{\rm Proj}),$$
were proved in the case that $R$ is a commutative noetherian ring with finite global dimension and $G$ is a group of type $\Phi_R$. The noetherian assumption therein can be removed.

\begin{corollary}\label{cor:triequ}
Let $R$ be a commutative ring with finite global dimension. If $G$ is either a group of type $\Phi_R$ or an ${\rm LH}\mathfrak{F}$-group of type $FP_\infty$, then the following categories are equivalent:
$$\begin{aligned} {\rm Ho}(\mathcal{F}ib) &\simeq  {\rm StMod}(RG) \simeq \underline{\mathcal{C}of}(RG) = \underline{\mathcal{GP}}(RG) \\
 &\simeq {\rm D}_{sg}(RG) \simeq {\rm K}_{tac}(RG\text{-}{\rm Proj}) = {\rm K}_{ac}(RG\text{-}{\rm Proj}).
\end{aligned}$$
\end{corollary}

\begin{proof}
The equivalences ${\rm Ho}(\mathcal{F}ib) \simeq {\rm StMod}(RG) \simeq \underline{\mathcal{C}of}(RG)$ hold for any group $G$ and any commutative ring $R$. Under the assumptions, we infer from Corollary \ref{cor:GP=Cof} that $\underline{\mathcal{C}of}(RG)$ and $\underline{\mathcal{GP}}(RG)$ coincide. Furthermore, it follows from Theorem \ref{thm:fGcd} and Corollary \ref{cor:Gcd=pdB2} that $RG$ is a Gorenstein regular ring. In this case, every acyclic complex of projective $RG$-modules is totally acyclic. Hence, ${\rm K}_{ac}(RG\text{-}{\rm Proj})$ agrees with ${\rm K}_{tac}(RG\text{-}{\rm Proj})$. The equivalence $\Omega: {\rm K}_{tac}(RG\text{-}{\rm Proj})\rightarrow \underline{\mathcal{GP}}(RG)$ follows from \cite[Theorem 4.16]{Bel00}. We infer from \cite[Theorem 6.9]{Bel00} or \cite[Theorem 3.3]{Chen11} that the natural functor $F: \underline{\mathcal{GP}}(RG)\rightarrow {\rm D}_{sg}(RG)$ is an equivalence.
\end{proof}

\begin{remark}
The singularity category was first studied by Buchweitz in his unpublished note \cite{Buc87} under the name of ``stable derived category''. In order to distinguish with the singularity category for finitely generated modules, the category is sometimes called a {\rm big singularity category} if the modules are not necessarily finitely generated; see e.g. \cite[pp. 205]{Chen11}.
\end{remark}

\vskip 10pt

\noindent {\bf Acknowledgements.}\quad
The author is grateful to Professor I. Emmanouil for sharing the preprint \cite{ET} and for helpful comments and suggestions
for an early version of the manuscript, and to Professors X.-W. Chen, D.-S. Li and W.-H. Qian for their helpful suggestions. 
The author is also grateful to the referee for helpful suggestions which result in a significant improvement of the paper. This work was supported by the National Natural Science Foundation of China (No. 11871125).

\vskip 10pt

{\footnotesize \noindent Wei Ren\\
 School of Mathematical Sciences, Chongqing Normal University, Chongqing 401331, PR China\\
 }


\begin{thebibliography}{9999}


\bibitem{ABS09} {\sc Asadollahi J., Bahlekeh A., Salarian S.} On the hierarchy of cohomological dimensions of groups.  \emph{J. Pure Appl. Algebra} {\bf 213}(9) (2009) 1795-1803.

\bibitem{AB69} {\sc Auslander M., Bridger M.} \emph{Stable module category}. Mem. Amer. Math. Soc. 94., 1969.

\bibitem{BDT09} {\sc Bahlekeh A., Dembegioti F., Talelli O.} Gorenstein dimension and proper actions. \emph{Bull. London. Math. Soc.}  {\bf 41} (2009) 859-871.

\bibitem{Bel00} {\sc Beligiannis A.} The homological theory of contravariantly finite subcategories: Auslander-Buchweitz contexts, Gorenstein categories and (co)stabilization. \emph{Commun. Algebra} {\bf 28} (2000) 4547-4596.

\bibitem{BM10} {\sc Bennis D., Mahdou N.} Global Gorenstein dimensions. \emph{Proc. Amer. Math. Soc.} {\bf 138} (2010) 461-465.

\bibitem{Ben97} {\sc Benson D.} Complexity and varieties for infinite groups I, II. \emph{J. Algebra} {\bf 193} (1997) 260-287, 288-317.

\bibitem{Bis21+} {\sc Biswas R.} Benson's cofibrants, Gorenstein projectives and a related conjecture. \emph{Proc. Edinburgh Math. Soc.} {\bf 64}(4) (2021) 779-799.

\bibitem{Bis21} {\sc Biswas R.} On some cohomological invariants for large families of infinite groups. \emph{New York J. Math.} {\bf 27} (2021) 818-839.

\bibitem{Bro82} {\sc Brown K.S.} \emph{Cohomology of Groups}. Graduate Texts in Mathematics 87, Springer, Berlin-Heidelberg-New York, 1982.

\bibitem{Buc87} {\sc R.O. Buchweitz}, Maximal Cohen-Macaulay Modules and Tate Cohomology over Gorenstein Rings. Unpublished manuscript, 155pp, 1987. Available at https://hdl.handle.net/1807/16682.

\bibitem{Bun10} {\sc B\"{u}hler T.} Exact Categories. \emph{Expo. Math.} {\bf 28}(1) (2010) 1-69.

\bibitem{Chen11} {\sc Chen X.-W.} Relative singularity categories and Gorenstein-projective modules. \emph{Math. Nachr.} {\bf 284} No. 2-3 (2011) 199-212.

\bibitem{CR21} {\sc Chen X.-W., Ren W.} Frobenius functors and Gorenstein homological properties. \emph{J. Algebra} {\bf 610} (2022) 18-37.

\bibitem{CZ07} {\sc Chen X.-W., Zhang P.} Quotient triangulated categories. \emph{Manuscripta Math.}  {\bf 123} (2007) 167-183.

\bibitem{Chr00} {\sc Christensen L.W.} \emph{Gorenstein Dimensions}. Lecture Notes in Math. Vol. 1747, Berlin: Springer-Verlag, 2000.

\bibitem{CK96} {\sc J. Cornick and P.H. Kropholler}, Homological finiteness conditions for modules over strongly group-graded rings. \emph{Math. Proc. Camb. Phil. Soc.} {\bf 120}(1) (1996) 43-54.

\bibitem{CK97} {\sc Cornick J., Kropholler P.H.}. On complete resolutions. \emph{Topology Appl.} {\bf 78}(3) (1997) 235-250.

\bibitem{DT08} {\sc Dembegioti F., Talelli O.} An integral homological characterization of finite groups. \emph{J. Algebra} {\bf 319} (2008) 267-271.

\bibitem{DT10} {\sc Dembegioti F., Talelli O.} A note on complete resolutions. \emph{Proc. Amer. Math. Soc.} {\bf 138}(11) (2010) 3815-3820.

\bibitem{Emm12} {\sc Emmanouil I.} On the finiteness of Gorenstein homological dimensions. \emph{J. Algebra} {\bf 372} (2012) 376-396.

\bibitem{ET14} {\sc Emmanouil I., Talelli O.} Finiteness criteria in Gorenstein homological algebra. \emph{Trans. Amer. Math. Soc.} {\bf 366} (2014) 6429-6351.

\bibitem{ET18} {\sc Emmanouil I., Talelli O.} Gorenstein dimension and group cohomology with group ring coefficients. \emph{J. London Math. Soc.} {\bf 97}(2) (2018) 306-324.

\bibitem{ET22} {\sc Emmanouil I., Talelli O.} On the Gorenstein cohomological dimension of group extensions. \emph{J. Algebra} {\bf 605} (2022) 403-428.

\bibitem{ET} {\sc Emmanouil I., Talelli O.} Characteristic modules and  Gorenstein (co-)homological dimension of groups. 
\emph{J. Pure Appl. Algebra} {\bf 229} (2025) 107830.

\bibitem{EEGR} {\sc Enochs E.E., Estrada S., Garc\'{\i}a-Rozas J.R.} Gorenstein categories and Tate cohomology on projective schemes. \emph{Math. Nachr.} {\bf 281}(4) (2008) 525-540.

\bibitem{EJ00} {\sc Enochs E.E., Jenda O.M.G.} \emph{ Relative Homological Algebra}. De Gruyter Expositions in Mathematics no. 30, New York: Walter De Gruyter, 2000.

\bibitem{GG87} {\sc Gedrich T.V., Gruenberg K.W.} Complete cohomological functors of groups. \emph{Topology Appl.} {\bf 25} (1987) 203-223.

\bibitem{Gil11} {\sc Gillespie J.} Model structures on exact categories. \emph{J. Pure Appl. Algebra} \textbf{215}(12) (2011) 2892-2902.

\bibitem{Hol04} {\sc Holm H.} Gorenstein homological dimensions. \emph{J. Pure Appl. Algebra} {\bf 189} (2004) 167-193.

\bibitem{Hov99}  {\sc Hovey M.} \emph{Model Categories}. Mathematical Surveys and Monographs vol. 63, American Mathematical Society, 1999.

\bibitem{Hov02} {\sc Hovey M.} Cotorsion pairs, model category structures and representation theory. \emph{Math. Z.} \textbf{241} (2002) 553-592.

\bibitem{Kro93} {\sc Kropholler P.H.} On groups of type ${\rm FP}_{\infty}$. \emph{ J. Pure Appl. Algebra} {\bf 90} (1993) 55-67.

\bibitem{MS19} {\sc Mazza N., Symonds P.} The stable category and invertible modules for infinite groups. \emph{Adv. Math.} {\bf 358} (2019) 106853, 26 pp.

\bibitem{Or04} {\sc Orlov D.} {\em Triangulated categories of singularities and D-branes in Landau-Ginzburg models}. Trudy Steklov Math. Institute {\bf 204} (2004) 240-262.

\bibitem{Qui67} {\sc Quillen D.G.} \emph{Homotopical Algebra}. Lecture Notes in Mathematics no. 43, Springer-Verlag, 1967.

\bibitem{Sta68} {\sc Stallings J.} On torsion-free groups with infinitely many ends. \emph{Ann. Math.} {\bf 88} (1968) 312-334.

\bibitem{Swan69} {\sc Swan R.G.} Groups of cohomological dimension one. \emph{J. Algebra} {\bf 12} (1969) 585-601.

\bibitem{Tal07} {\sc Talelli O.} On groups of type $\Phi$. \emph{Arch. Math.} {\bf 89}(1) (2007) 24-32.

\bibitem{Tal14} {\sc Talelli O.} On the Gorenstein and cohomological dimension of groups. \emph{Proc. Amer. Math. Soc.} {\bf 142}(4) (2014) 1175-1180.

\bibitem{Tal17} {\sc Talelli O.} On characteristic modules of groups. \emph{Geometric and Cohomological Group Theory}, P.H.
Kropholler et al. (Eds.), London Math. Soc. Lecture Note Ser. 444, Cambridge Univ. Press (2017) 172-181.

\end{thebibliography}
\end{document}